\documentclass[11pt,a4paper]{amsart}
\usepackage[utf8]{inputenc}
\usepackage[english]{babel}
\usepackage[left=2.5cm,right=2.5cm,top=2cm,bottom=2cm]{geometry}
\usepackage{graphicx}
\usepackage{amssymb}
\usepackage{amsmath}
\usepackage{amsthm}
\usepackage{algorithm}
\usepackage{algorithmic}
\usepackage{comment}
\usepackage{pdfpages}
\usepackage{appendix}
\usepackage{stmaryrd}
\usepackage{tikz}
\usepackage{mathtools}
\usepackage{bbm}
\usepackage[maxbibnames=50, style=alphabetic]{biblatex}

\usepackage{csquotes}
\usepackage{hyperref}
\hypersetup{
  colorlinks   = true, 
  linkcolor    = black,
  citecolor    = black ,
  urlcolor = black     
}

\providecommand{\E}{\mathbb E}
\providecommand{\ed}{\mathrm e}
\providecommand{\diff}{\mathrm d}
\providecommand{\Tr}{\mathrm {Tr}}
\providecommand{\prob}{\mathbb P}

\providecommand{\D}{\mathbb D}
\providecommand{\C}{\mathbb C}

\numberwithin{equation}{section}
\newtheorem{theorem}{Theorem}[section]
\newtheorem{proposition}[theorem]{Proposition}

\newtheorem{lemma}[theorem]{Lemma}
\theoremstyle{definition}
\newtheorem{definition}[theorem]{Definition}

\newtheorem{remark}[theorem]{Remark}

\title{Characteristic polynomial of Generalized Ewens random permutations}
\author{Quentin François}

\address
{Quentin François: CEREMADE, CNRS, Université Paris-Dauphine, Université PSL, 75016 Paris, France
\& DMA, École normale supérieure, Université PSL, CNRS, 75005 Paris, France.}

\email{\href{mailto:quentin.francois@dauphine.psl.eu}{quentin.francois@dauphine.psl.eu}}

\begin{document}

\begin{abstract}
  We prove the convergence of the characteristic polynomial for random permutation matrices 
  sampled from the generalized Ewens distribution. Under this distribution, the measure of 
  a given permutation depends only on its cycle structure with weights 
  assigned to each cycle length. 
  The proof is based on uniform control of the characteristic polynomial 
  using results from the singularity analysis of generating functions, 
  together with the convergence of traces to explicit random variables 
  expressed via a Poisson family. 
  The limit function is the exponential of a Poisson series 
  which has already appeared in the case of uniform permutation matrices. 
  It is the Poisson analog of the Gaussian Holomorphic Chaos, 
  related to the limit of characteristic polynomials 
  for other matrix models such as Circular Ensembles, i.i.d.\ matrices, and Gaussian elliptic matrices.
\end{abstract}

\maketitle

\section{Introduction}

The study of the characteristic polynomial of random matrices 
has gained importance in the context of random analytic functions 
\cite{Coste_Lambert_Zhu,Lambert_Paquette} and random 
fields \cite{Hughes_Keating_Neil, Rider_Virag}.
Coefficients of characteristic polynomials exhibit some 
combinatorial structure as shown by Diaconis and Gamburd \cite{Diaconis_Gamburd} 
in the case of random unitary matrices sampled from the Haar measure. 
Instead of individual coefficients, one can consider the characteristic 
polynomial as a random variable in the space of analytic functions. 
The behavior of the characteristic polynomial outside of the support of the 
limit eigenvalue distribution is of particular interest. 
One main motivation for the latter is the analysis of outliers 
with respect to the global behavior of eigenvalues given by the 
convergence of the empirical eigenvalue distribution. 
This approach was followed by Bordenave, Chafaï and García-Zelada 
\cite{Bordenave_Chafai_Garcia}, proving 
the convergence of the characteristic polynomial of Girko matrices, that is, 
matrices with i.i.d.\ centered entries under a universal second order moment condition. 
This allowed them to prove a convergence of the spectral radius for such matrices 
to one which could not be obtained from the convergence of the eigenvalue 
measure to the uniform law on the unit disk.
Their work was inspired by the results of Basak and Zeitouni \cite{Basak_Zeitouni} who 
studied outliers for eigenvalues of Toeplitz matrices. \\
\\
The limit function obtained for the 
characteristic polynomial of Girko matrices in \cite{Bordenave_Chafai_Garcia} 
involves the 
exponential of a Gaussian analytic function. 
Such an expression is an example of a Gaussian log-correlated field, 
see \cite{Najnudel_Paquette_Simm} and reference therein. The corresponding random distribution 
was introduced as the holomorphic multiplicative chaos. It 
arises as the limit of the characteristic polynomial for Circular-$\beta$ Ensembles 
\cite{Chhaibi_Najnudel} and its Fourier coefficients are related to the 
enumeration of combinatorial objects called magic squares. 
The holomorphic multiplicative chaos also appears as the limit of the 
characteristic polynomial of Gaussian elliptic matrices which interpolate 
between Ginibre and GUE matrices \cite{Francois_Garcia}. This form of the 
limit was proved to be universal in \cite{Bordenave_Chafai_Garcia} for Girko 
matrices and is conjectured to hold for a larger class of elliptic matrices 
interpolating between i.i.d.\ and Hermitian models. \\
\\
Coste \cite{Coste_Bernoulli} considered the case of non-centered 
entries following Bernoulli distribution. The characteristic polynomial for 
such matrices in the sparse regime 
was shown to converge towards a random analytic function 
expressed as the exponential of a Poisson series. This form is the Poisson 
analog of the holomorphic multiplicative chaos and has connections to the enumeration of 
multiset partitions. The same function was also identified as the 
limit of characteristic polynomial for sums of random permutation matrices 
where the permutation follows the uniform distribution by Coste, Lambert and Zhu 
\cite{Coste_Lambert_Zhu}. The authors raised the question of extending 
their results to other measures on the space of permutations notably 
to the Ewens measure \cite{Ewens}, a measure in which the weight of a permutation depends 
only on its cycle structure. 
\\ 
\\
The goal of this paper is to answer the previous question on the 
convergence of the characteristic polynomial in the context of 
generalized Ewens distributed permutations, which encompasses the Ewens, 
and thus, uniform cases. The generalized Ewens distribution 
was introduced by Nikeghbali and Zeindler \cite{Nikeghbali_Zeindler} as a 
generalization of the classical Ewens distribution by assigning different weights 
to each cycle lengths. Following the results of Chhaibi, Najnudel and Nikeghbali 
\cite{Chhaibi_Najnudel_Nikeghbali} on the characteristic polynomial of Haar unitary matrices, Bahier 
\cite{Bahier} showed the convergence of the characteristic polynomial of Ewens 
permutation matrices at a microscopic scale around one 
and near irrational angles on the unit circle. 
Here, we consider the characteristic polynomial in a different regime namely inside 
the open unit disk where there are no eigenvalues. \\
\\
For $n \geqslant 1$, we denote by $S_n$ the group of permutations of $\{1, \dots, n\}$. 

\begin{definition}[Generalized Ewens measure, \cite{Nikeghbali_Zeindler}]
    Let $\Theta = (\theta_k)_{k \geqslant 1}$ be a sequence of positive real numbers. 
    For $n \geqslant 1$, the \textbf{generalized Ewens measure} 
    is the probability measure $\diff \prob_n^\Theta $ on $S_n$ 
    defined by 
    \begin{equation}
        \label{eq:def_gen_ewens}
        \diff \prob_n^\Theta [\sigma] 
        = \frac{1}{n! h_n^{\Theta}} \prod_{k=1}^n \theta_k^{C_k(\sigma)}, 
    \end{equation}
    where for a permutation $\sigma \in S_n$ and $k \geqslant 1$, $C_k(\sigma)$ is the number 
    of cycles of $\sigma$ with length $k$.
\end{definition}

\noindent
The Ewens measure corresponds to the case where the sequence $\Theta$ is constant equal to $\theta > 0$ in 
which case $h_n^{\Theta} = \binom{\theta+n-1}{n}$. The uniform measure on $S_n$ corresponds to the 
Ewens distribution with parameter $\theta = 1$. 
From the sequence $\Theta = (\theta_k)_{k \geqslant 1}$, 
one defines the formal power series as in \cite{Nikeghbali_Zeindler},
\begin{equation}
    \label{eq:g_and_G}
    g_\Theta(z) = \sum_{k \geqslant 1} \frac{\theta_k}{k}z^k 
    \text{ and } G_\Theta(z) = \exp(g_\Theta(z)) \ . 
\end{equation}
For the Ewens measure of parameter $\theta$, $g_\Theta$ and 
$G_\Theta$ are holomorphic in $\D$ with 
$g_\Theta(z) = - \theta \log(1 - z)$ and $G_\Theta(z) = (1-z)^{-\theta}$. 
By \cite[Lemma 2.6]{Hughes_Najnudel_Nikeghbali_Zeindler}, one has 
\begin{equation*}
    G_\Theta(z) = \sum_{n \geqslant 0} h_n^{\Theta} z^n,
\end{equation*}
where $h_n^{\Theta}$ are the constants in the definition of the generalized Ewens 
distribution \eqref{eq:def_gen_ewens}. 
\\
\\
In this paper, we consider characteristic polynomials of random matrices associated 
to random permutations sampled from the generalized Ewens distribution \eqref{eq:def_gen_ewens}. 
Since permutations $\sigma \in S_n$ can be viewed as 
permutation matrices of size $n$, we say that $A_n$ follows the 
generalized Ewens distribution if it is the matrix obtained from a 
permutation $\sigma$ sampled from \eqref{eq:def_gen_ewens}.
The characteristic polynomial $p_n(z) = \det(1 - zA)$ of a permutation 
matrix $A$ can be expressed as 
\begin{equation}
    p_n(z) = \prod_{k=1}^n (1 - z^k)^{C_k^{(n)}},
\end{equation}
where $\left(C_k^{(n)}, 1 \leqslant k \leqslant n \right)$ are the cycle lengths 
of the associated random permutation. 
Note that the eigenvalues of $A_n$ are explicit and are given by roots of 
unity located on the unit circle.
Our result aims at showing the convergence of $(p_n)_{n \geqslant 1}$ 
as a sequence of random holomorphic functions defined on the unit disk. 
As in \cite{Bordenave_Chafai_Garcia,Francois_Garcia,Coste_Bernoulli} and \cite{Coste_Lambert_Zhu}, 
we consider the limit of the characteristic polynomial in the region outside 
of the eigenvalue support, namely $p_n(z) = z^n \det(z^{-1} - A_n)$ so that 
for $z \in \D$, $z^{-1}$ lies outside of the unit circle and $\det(z^{-1} - A_n)$ 
does not vanish.

\section{Main result}

\subsection{Convergence of the characteristic polynomial}

For $n \geqslant 1$ and $\Theta = (\theta_k)_{k \geqslant 1}$ as above, 
we consider $A_n$ the random matrix associated to a 
permutation $\sigma$ sampled from \eqref{eq:def_gen_ewens}. 
In this paper, we consider characteristic polynomial 
\begin{equation}
    p_n(z) = \det (1 - z A_n) 
\end{equation}
inside the unit disk $z \in \D =\{ x \in \C: |x| < 1 \}$. 
Let us denote by $\mathcal{H}(\D)$ the space of holomorphic functions 
on $\D$ endowed with the topology of convergence on compact subsets of $\D$.
Our main result is the convergence of $p_n$ as a random variable in $\mathcal{H}(\D)$ 
in law towards a limit function $F \in \mathcal{H}(\D)$. 
The above convergence holds for parameters $\Theta$ such that the generating 
series $g_\Theta$ satisfies some conditions that we now define as
Definition \ref{def:log_class} which is an adaptation of a definition given 
in Section 5.2.1 of \cite{hwang1994}. One can also find it as Definition 
2.9 in \cite{Hughes_Najnudel_Nikeghbali_Zeindler} 
or Definition 2.8 in \cite{Nikeghbali_Zeindler}. 
Throughout this work, $\arg(z)$ denotes the principal argument of $z$ 
taking values in $(-\pi, \pi]$.

\begin{definition}[Logarithmic class function]
    \label{def:log_class}
    A function $g$ is said to be in $F(r, \gamma, K)$ 
    for $r > 0$, $\gamma \geqslant 0$ and $K \in \C$ if 
    \begin{itemize}
        \item There exists $R > r$ and $\phi \in (0, \pi/2)$ such that $g$ is 
        holomorphic in $\Delta(r, R, \phi) \setminus \{r \}$ where 
        $\Delta(r, R, \phi) = \{ z \in \C: |z| \leqslant R, |\arg(z-r)| \geqslant \phi \}$.
        \item As $z \rightarrow r$, $g(z) = -\gamma \log(1 - z/r) + K + O(z-r)$. 
    \end{itemize}
\end{definition}

\noindent
In the case of the Ewens measure of parameter $\theta$, we have 
$g_\Theta(z) = -\theta \log(1 - z)$ so that $g_\Theta \in F(1, \theta, 0)$. 
Note that if $\gamma > 0$, the parameter $r$ is unique. 
\\
\\
Our main result is Theorem \ref{th:conv} which gives the convergence of the 
characteristic polynomial towards a limit function for sequences $\Theta$ such 
that $g_\Theta$ satisfies the conditions of Definition \ref{def:log_class}.

\begin{theorem}[Convergence of the characteristic polynomial]
    \label{th:conv}
    Let $\Theta = (\theta_k)_{k \geqslant 1}$ be a sequence of 
    positive real numbers such 
    that $g_\Theta \in F(r, \gamma, K)$ for some $r > 0$, 
    $\gamma > 0$ and $K \in \C$.
    We have the convergence in law, for the topology of 
    local uniform convergence in $\D$:
    \begin{equation}
        p_n \overunderset{\diff}{n \to \infty}{\longrightarrow} F ,
    \end{equation}
    where
    \begin{equation}
        F(z) = \exp \left( - \sum_{k \geqslant 1} 
        \frac{z^k}{k} X_k \right), 
        \ X_k = \sum_{\ell | k} \ell Y_\ell \ ,
    \end{equation}
    with $(Y_\ell)_{\ell \geqslant 1}$ 
    independent Poisson random variables 
    with parameter $\frac{\theta_\ell}{\ell}r^\ell$.
\end{theorem}

\noindent
The previous theorem gives in particular the convergence of the characteristic 
polynomial for Ewens permutation matrices. Indeed, for constant $\theta$, the 
function $g_\Theta \in F(1, \theta, 0)$ so that $p_n$ converges towards 
the limit function as conjectured in \cite{Coste_Lambert_Zhu}. 

\begin{figure}[ht]
    \centering
    \begin{minipage}{0.45\textwidth}
        \centering
        \includegraphics[scale=0.23]{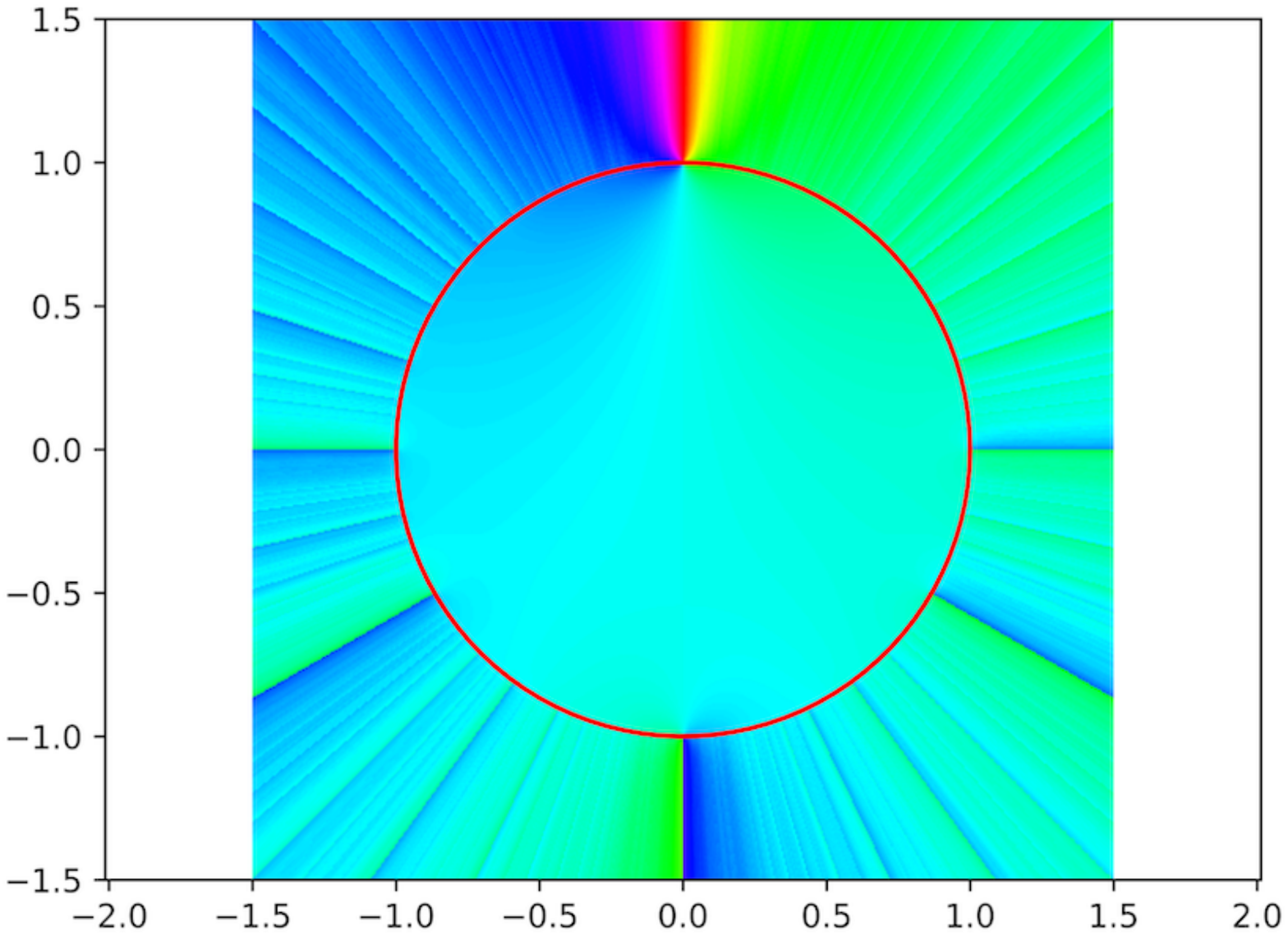}
    \end{minipage}
    \hfill
    \begin{minipage}{0.45\textwidth}
        \centering
        \includegraphics[scale=0.23]{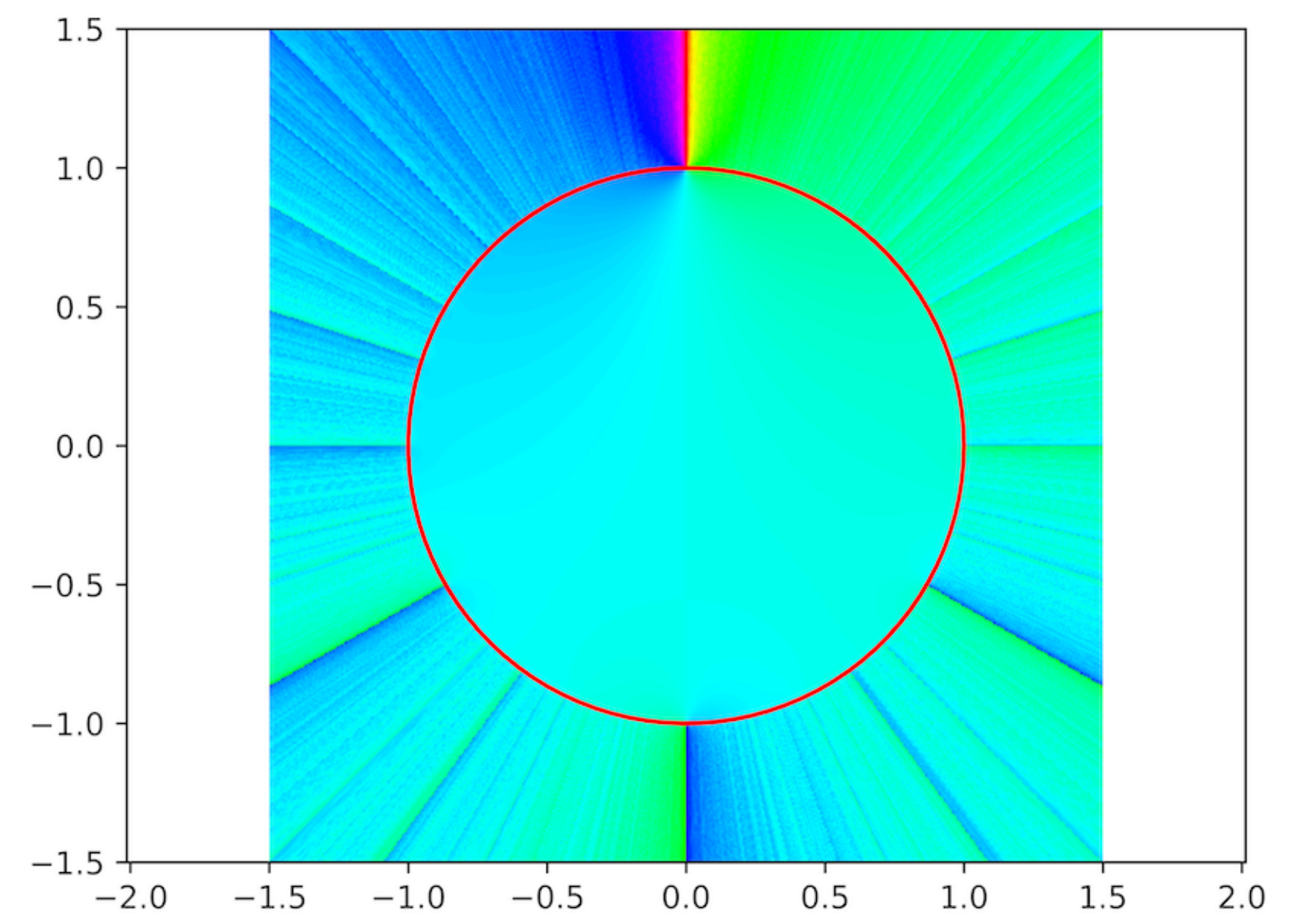}
    \end{minipage}
    \caption{Phase portrait of $p_n$ for an Ewens matrix of 
    size $n=10000$ with parameter $\theta=100$ (left) and phase portrait of the 
    limit function with same parameter (right). The unit circle is represented in red.}
    \label{fig:both}
\end{figure}

\begin{remark}[Outside region]
    Theorem \ref{th:conv} deals with the convergence in law for $z \in \D$ so that 
    $p_n(z) = \det(1 - zA_n)$ does not vanish as eigenvalues of $A_n$ are located on the 
    unit circle. One can extend the previous to values of $p_n(z)$ for $z$ in 
    $\C \setminus \overline{\D} = \{ z \in \C: |z| > 1 \}$ under suitable normalization. 
    Indeed, notice that the generalized Ewens distribution \eqref{eq:def_gen_ewens} 
    is invariant under inversion, 
    that is, if $\sigma$ has distribution \eqref{eq:def_gen_ewens} then so does $\sigma^{-1}$ 
    as they both have the same cycle structure. Thus, $A_n = A_n^{-1}$ in law.
    Furthermore, for $z \in \C \setminus \overline{\D}$,
    $\det(1 - zA_n) = (-z)^n \det(A_n) \det(1 - z^{-1}A_n^{-1})$ so that 
    if $\tilde{p}_n(z) =  \frac{p_n(z)}{(-z)^n \det(A_n)}$ 
    and $\iota(z) = \frac{1}{z}$, Theorem \ref{th:conv}, 
    gives the convergence in 
    law on $\C \setminus \overline{\D}$ of $\tilde{p}_n$ to $F \circ \iota$. \\
\end{remark}

\subsection{Method of proof}
The proof of Theorem \ref{th:conv} relies on the same structure as in \cite{Bordenave_Chafai_Garcia}, 
which is recalled in Lemma \ref{lemma:conditions_a_b}. 
It is a consequence of the general fact stated in \cite{Shirai} that 
a tight sequence of holomorphic functions in $\mathcal{H}(\D)$ whose coefficients 
convergence in law for finite dimensional distributions converges to a random 
analytic function.
We first show that the sequence $(p_n)_{n \geqslant 1}$ is tight which is Theorem \ref{th:tightness} 
proved in Section \ref{sec:tightness}. 
Here, tightness is archieved by a uniform control of the second moment of $p_n$. 
This control relies on results from Hwang \cite{hwang1994} on singularity analysis 
for generating functions. 
The finite dimensional convergence of coefficients is obtained by 
showing the convergence of traces of powers, 
see the discussion above Theorem \ref{th:conv_of_traces}.
The convergence of traces for generalized Ewens matrices 
was done in \cite{Hughes_Najnudel_Nikeghbali_Zeindler} and 
\cite{Nikeghbali_Zeindler}. We recall their results 
in Section \ref{sec:conv_traces} where Theorem \ref{th:conv_of_traces} is proved. 
From these two results, one is able to derive the convergence of $p_n$ 
towards a random analytic function $F$. 
The fact that $F$ coincides with the exponential of a Poisson series 
is the purpose of Theorem \ref{th:poisson_expression} which is 
proved in section \ref{sec:poisson_expression}.
In the rest of the paper, we assume that $\Theta$ is fixed and 
we write $g$ and $G$ for the functions defined in \eqref{eq:g_and_G} 
for notation convenience.

\begin{remark}[Extension to sums of permutations]
    In the case of a sum of $d \geqslant 1$ Ewens random matrices with constant parameter $\theta$, 
    the finite dimensional convergence of traces is given as in Theorem 3.2 of \cite{Coste_Lambert_Zhu}, 
    with parameters of Poisson variables depending on $\theta$, see the discussion in Appendix B therein. 
    An extension of this argument to the generalized Ewens distribution could identify the limit 
    random variables for more than one factor.   
    However, the tightness criterion in \cite{Coste_Lambert_Zhu} relies on exchangeability, both for 
    $d=1$ and $d > 1$ uniform permutations, which no longer holds for non-uniform permutations. Here, we are able to 
    show tightness for $d=1$ in the non-uniform case by relating the characteristic polynomial 
    to generating functions, see Proposition \ref{prop:second_order_moment}. The extension 
    to $d > 1$ factors is more involved as no such relation to generating functions is known. 
\end{remark}

\section{Proof of Theorem \ref{th:conv}}

Recall that $\mathcal{H}(\D)$ denotes the space of analytic functions on $\D$ 
endowed with the topology of local uniform convergence. In order to show the 
convergence in law of a sequence $(f_n)_{n \geqslant 1}$ in $\mathcal{H}(\D)$, we 
rely on Lemma \ref{lemma:conditions_a_b} which is close to Proposition 2.5 in \cite{Shirai}. 
It is also stated as Lemma 3.2 in \cite{Bordenave_Chafai_Garcia} and proved therein.

\begin{lemma}[Tightness and convergence of coefficients imply convergence of functions]
	\label{lemma:conditions_a_b}
	Let $\{f_n \}_{n \geqslant 1}$ be a sequence of random elements in $\mathcal{H}(\D)$ and denote 
	the coefficients of $f_n$ by $(\xi_k^{(n)} )_{k \geqslant 0}$ so that for all $z \in \D$, 
	$f_n(z) = \sum_{k \geqslant 0} \xi_k^{(n)} z^k$. Suppose also that the following conditions hold.
	\begin{itemize}
		\item[$(a)$] The sequence $\{f_n \}_{n \geqslant 1}$ is
		a tight sequence of random elements of $\mathcal{H}(\D)$.
		\item[$(b)$] 
		There exists a sequence $(\xi_k)_{k \geqslant 0}$ of 
		random variables such that,
		for every $m \geqslant 0$, the vector $\left(\xi_0^{(n)}, \dots, \xi_m^{(n)}\right)$
		converges in law as $n \to \infty$ to $(\xi_0, \dots, \xi_m)$.
	\end{itemize}
	Then, $f(z) = \sum_{k \geqslant 0} \xi_k z^k$ is a well-defined function in $\mathcal{H}(\D)$ and $f_n$
	converges in law towards $f$ in $\mathcal{H}(\D)$ for the topology of local uniform convergence.
\end{lemma}

\noindent
We thus need to show that the sequence $(p_n)_{n \geqslant 1}$ is tight and then 
study the limit of finite dimensional distributions for its coefficients. 
The first part is given by Theorem \ref{th:tightness} 
which is proved in section \ref{sec:tightness}. 

\begin{theorem}[Tightness]
    \label{th:tightness}
    The sequence $(p_n)_{n \geqslant 1}$ is tight in $\mathcal{H}(\D)$.
\end{theorem}

\noindent
It remains to study the coefficients of $p_n$. Let us write 
\begin{equation*}
    p_n(z) = 1 + \sum_{k =1}^{n} (-z)^k \Delta_k(A_n)
\end{equation*}
where $\Delta_k(A)$ is the coefficient of $z^k$ in $\det(1+zA)$. 
Coefficients $\Delta_k(A_n)$ can be expressed via traces 
$\left( \Tr \left[A_n^\ell \right], 1 \leqslant \ell \leqslant k \right)$ so that
\begin{equation*}
    \Delta_k(A_n) = \frac{1}{k!} P_k 
    \left(\Tr \left[A_n^1 \right], \dots, \Tr \left[A_n^k \right] \right)
\end{equation*}
where the polynomials $P_k$ do not depend on $n$. In order to study the convergence 
in law of coefficients $\left(\Delta_1(A_n), \dots, \Delta_k(A_n) \right)$, it suffices to study 
the convergence of traces $\left(\Tr[A_n^1], \dots, \Tr[A_n^k] \right)$ which is given by 
Theorem \ref{th:conv_of_traces}. Recall that $r$ denotes the radius of convergence of $g$, 
see Definition \ref{def:log_class}.

\begin{theorem}[Convergence of coefficients]
    \label{th:conv_of_traces}
    For $k \geqslant 1$, we have the convergence in law
    \begin{equation}
        \left( \Tr[A_n], \dots, \Tr \left[A_n^k \right]\right)
        \overunderset{\diff}{n \to \infty}{\longrightarrow} 
        (X_1, \dots, X_k) \ ,
    \end{equation}
    where 
    \begin{equation}
        X_k = \sum_{\ell | k} \ell Y_\ell
    \end{equation}
    with $(Y_\ell, \ell \geqslant 0)$ are independent Poisson random variables 
    with parameter $\frac{\theta_\ell}{\ell}r^\ell$.
\end{theorem}

\noindent
Thanks to Lemma \ref{lemma:conditions_a_b}, Theorem \ref{th:tightness} 
and Theorem \ref{th:conv_of_traces}, we derive that $p_n$ converges 
towards the random analytic function $F \in \mathcal{H}(\D)$ given by
\begin{equation*}
    F(z) = 1 + \sum_{k \geqslant 1} \frac{(-z)^k}{k!} 
    P_k(X_1, \dots, X_k) \ .
\end{equation*}

\noindent
To obtain the expression of Theorem \ref{th:conv}, we rely on Theorem 
\ref{th:poisson_expression}, proved in Section \ref{sec:poisson_expression}
which yields the desired expression and ends the proof of Theorem \ref{th:conv}.

\begin{theorem}[Poisson expression for $F$]
    \label{th:poisson_expression}
    Under the assumptions of Theorem \ref{th:conv}, one has 
    that for every $z \in \D$, almost surely, 
    \begin{equation}
        F(z) = \exp(-f(z))
    \end{equation}
    where 
    \begin{equation*}
        f(z) = \sum_{k \geqslant 1} \frac{X_k}{k} z^k , 
    \end{equation*}
    and where $(X_k)_{k \geqslant 1}$ are 
    defined as in Theorem \ref{th:conv}.
\end{theorem}

\section{Tightness: proof of Theorem \ref{th:tightness}}
\label{sec:tightness}

The goal of this section is to prove Theorem \ref{th:tightness}.
We start by Lemma \ref{lem:red_unif_control} which reduces the tightness of 
a sequence of a functions $(f_n)$ to proving tightness of their local supremum.
This lemma corresponds to Proposition 2.5 of \cite{Shirai}.

\begin{lemma}[Uniform control]
    \label{lem:red_unif_control}
    Let $(f_n)_{n \geqslant 1}$ be  
    random variables in $\mathcal{H}(\D)$. 
    If for every compact $K \subset \D$, 
    the sequence 
    $\left(\sup_{z \in K} |f_n(z)| \right)_{n \geqslant 1}$ 
    is tight, then $(f_n)_{n \geqslant 1}$ is tight. \\
\end{lemma}

\noindent
It therefore suffices to show that 
$(\sup_{z \in K} |f_n(z)|)_{n \geqslant 1}$ is tight for every compact $K \subset \D$. 
A sufficient criterion is to bound the second moment 
$ \left(\E \left[ \left(\sup_{z \in K} |f(z)| \right)^2 \right] \right)_{n \geqslant 1}$.
Using Lemma 2.6 in \cite{Shirai}, this is equivalent to prove that 
$\left(\sup_{z \in K} \E[|f_n(z)|^2] \right)_{n \geqslant 1}$ is bounded.
We will show this for the sequence $(p_n)_{n \geqslant 1}$ of 
characteristic polynomials by giving 
a uniform control of the second moment of $p_n$ in 
Proposition \ref{prop:second_order_moment}. 
This control relies on the same asymptotic as in Corollary 3.8 of 
\cite{Hughes_Najnudel_Nikeghbali_Zeindler}, where we explicit 
the fact that the error is uniform for $z$ in
compact subsets of $\D$. \\
\\
Recall that the functions $g$ and $G$ are defined in \eqref{eq:g_and_G} by 
$g= \sum_{k \geqslant 1} \frac{\theta_k}{k}z^k \text{ and } G(z) = \exp(g(z))$ 
for $|z| < r$ where $r$ is the convergence radius of $g$. 
For $\delta > 0$, let us denote the open disk of radius $\delta$ by 
\begin{equation*}
    \D_\delta = \{ z \in \C \mid |z| < \delta \} .
\end{equation*}

\begin{proposition}[Second moment control]
    \label{prop:second_order_moment}
    Assume that $g \in F(r, \gamma, K)$ for some 
    $r > 0$, $\gamma \geqslant 0$ and $K \in \C$. 
    For $\delta \in (0, 1)$ and $z \in \D_\delta$, 
    one has the asymptotic expansion
    \begin{equation}
        \label{eq:asyp_sec_moment}
        \E [|p_n(z)|^2] = \frac{G(r|z|^2)}{G(rz) G(r\overline{z})} 
        + O \left(\frac{1}{n} \right) ,
    \end{equation}
    where the $O$ term holds uniformly in $z \in \D_\delta$.
\end{proposition}

\begin{proof}
    Let us fix $\delta \in (0,1)$. For $z \in \D_\delta$, one has using Corollary 3.6 of 
    \cite{Hughes_Najnudel_Nikeghbali_Zeindler}
    \begin{equation}
        \label{eq:generating_hwang}
        \sum_{n \geqslant 0} t^n h_n^\Theta \E[|p_n(z)|^2] = \exp(g(t)) S_z(t)
    \end{equation}
    where $h_n^\Theta$ are the coefficients in \eqref{eq:def_gen_ewens} and
    \begin{equation}
        \label{eq:def_S_n_z}
        S_z(t) =  \frac{G(t|z|^2)}{G(tz) G(t\overline{z})} \ .
    \end{equation}
    For every $z \in \D_\delta$, the function $t \mapsto G(zt)$ is analytic for 
    $|t| \leqslant r+\epsilon_1$ for some $\epsilon_1 > 0$ such that $\delta(r+\epsilon_1) < r$ 
    since $G(u) = \exp(g(u))$ and since that $g$ is analytic in $\D_r$. 
    Therefore, for every $z \in \D$, $t \mapsto S_z(t)$ is analytic for $|t| \leqslant r+\epsilon_1$.
    \\
    \\
    By assumption, $g$ is analytic for $t \in \Delta(r, r+\epsilon_2, \phi)$ for 
    some $\epsilon_2 > 0$ and $0 < \phi < \frac{\pi}{2}$. 
    Set $R = r + \min(\epsilon_1, \epsilon_2)$ and $\xi > 0$ 
    such that $r\ed^\xi < R$.
    The key idea, as done in Section 5.3.2 of \cite{hwang1994}, is to 
    compute asymptotic of $\E [|p_n(z)|^2]$ using \eqref{eq:generating_hwang}. 
    Let us define contours 
    \begin{align*}
        \Gamma &= \{ t \in \C \mid |t-1| = r (\ed^\xi-1), |\arg(t-r)| \geqslant \phi \}, \\
        \Gamma' &= \{ t \in \C \mid |t| = r \ed^\xi, |\arg(t-r)| \geqslant \phi \},
    \end{align*}
    oriented as in Figure \ref{fig:contours}.
    \begin{figure}[h]
        \centering
        \includegraphics[scale=1]{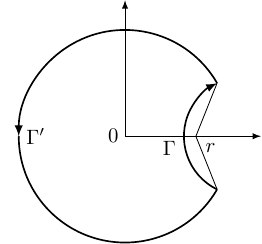}
        \caption{Integration contours $\Gamma$ and $\Gamma'$.}
        \label{fig:contours}
    \end{figure}
    We have
    \begin{equation*}
        h_n^\Theta \E[|p_n(z)|^2] = \frac{1}{2i \pi} \int_\Gamma \frac{ \exp(g(t)) S_z(t)}{t^{n+1}} dt 
        + \frac{1}{2i \pi} \int_{\Gamma'} \frac{ \exp(g(t)) S_z(t)}{t^{n+1}} dt \ .
    \end{equation*}
    For the integral over $\Gamma'$, we may use \eqref{eq:def_S_n_z} to derive the bound 
    \begin{equation}
        \label{eq:unif_bound_S}
        |S_z(t)| \leqslant \frac{\sup_{|u| \leqslant \delta^2r \ed^\xi} |G(u)|}
        {(\inf_{|u| \leqslant \delta r \ed^\xi} |G(u)|)^2} = C \ ,
    \end{equation}
    where $C$ does not depend on $z$.  The contribution of this integral 
    is thus $O(r^{-n}\ed^{-n\xi})$, where the $O$ term is uniform in $z$. 
    Let us consider the integral over $\Gamma$. Apply the change of 
    variables $t = r \ed^{- u}$, mapping $\Gamma$ to $\Gamma_0$, so that 
    \begin{equation*}
        \int_\Gamma \frac{ \exp(g(t)) S_z(t)}{t^{n+1}} dt 
        = r^{-n} \int_{\Gamma_0} \ed^{nu} \exp(g(r \ed^{-u})) S_z(r \ed^{-u}) du 
        = r^{-n} \ed^{K} \int_{\Gamma_0} u^{-\gamma} \ed^{nu} U_z(u) du
    \end{equation*}
    where 
    \begin{equation*}
        U_z(u) = \left( \frac{u}{1 - \ed^{-u}} \right)^\gamma 
        \ed^{ H(1 - \ed^{-u}) - K} S_z(r \ed^{-u}) \ , 
    \end{equation*}
    and where $H$ is the function defined by 
    \begin{equation*}
        g(r \ed^{-u}) = - \gamma \log \left( 1 - \ed^{-u} \right) 
        + H \left(1 - \ed^{-u} \right) ,
    \end{equation*}
    so that $H \left(1 - \ed^{-u} \right) = K + O \left(r(\ed^{-u} - 1) \right)$ since $g \in F(r, \gamma, K)$.
    Using the asymptotic development of $g$, one has 
    \begin{equation*}
        U_z(u) = (1 + O(|u|)) S_z(r \ed^{-u})
    \end{equation*}
    where the $O(|u|)$ is independent of $z$ as it only depends on $g$. 
    Moreover, 
    \begin{equation*}
        S_z(r\ed^{-u}) = S_z(r) + r(\ed^{-u} - 1) \cdot O(1) \ ,
    \end{equation*}
    where the $O(1)$ term is independent of $z$ as it can be 
    taken as as $ \sup_{|u| \leqslant r} |S_z'(u)|$, 
    which can be bounded uniformly with respect to $z$ by bounding values of $G$ 
    and $G'$ in $\D_{r \delta}$ in a similar fashion as in \eqref{eq:unif_bound_S}. 
    We derive that $U_z(u) = S_z(r) + O(|u|)$ uniformly in $z \in \D_\delta$. 
    \\
    \\
    One can therefore write the remaining integral as 
    \begin{equation*}
        \int_\Gamma \frac{ \exp(g(t)) S_z(t)}{t^{n+1}} dt = 
        r^{-n} \ed^{K} S_z(r) \int_{\Gamma_0} u^{-\gamma} \ed^{nu} du 
        + R(n) \ ,
    \end{equation*}
    where $R(n) = O \left( r^{-n} \int_{\Gamma_0} |u|^{1-\gamma} \ed^{n \Re u} |du|  \right)$ 
    and where the constant does not depend on $z \in \D_\delta$.
    Asymptotics of $R(n)$ and $ \int_{\Gamma_0} u^{-\gamma} \ed^{nu} du$ 
    are computed 
    in \cite{hwang1994}, see Section 5.3.2, eq. (5.7) therein, which yields 
    \begin{equation*}
         h_n^\Theta \E[|p_n(z)|^2] = \frac{\ed^K n^{\gamma - 1}}{r^n \Gamma(\gamma)} S_z(r) 
         + O \left(r^{-n} n^{\gamma - 2}\right) \ .
    \end{equation*}
    Using that $h_n^\Theta = \frac{\ed^K n^{\gamma -1}}{r^n \Gamma(\gamma)} 
    (1 + O \left( \frac{1}{n} \right))$, given by Lemma 2.13 of 
    \cite{Nikeghbali_Zeindler}, we obtain the asymptotic
    \begin{equation*}
        \E[|p_n(z)|^2] = S_z(r) + O \left( \frac{1}{n} \right)
    \end{equation*}
    with an error term uniform in $z \in \D_\delta$ as desired.
\end{proof}

\noindent
Proposition \ref{prop:second_order_moment}, 
implies that $\left(\sup_{z \in K} \E[|f_n(z)|^2] \right)_{n \geqslant 1}$ is bounded 
for every compact $K \subset \D$, so that the sequence $(p_n)_{n \geqslant 1}$ 
is tight thanks to Lemma \ref{lem:red_unif_control} yielding Theorem \ref{th:tightness}.

\section{Convergence of traces: proof of Theorem \ref{th:conv_of_traces}}
\label{sec:conv_traces}

The purpose of this section is to prove Theorem \ref{th:conv_of_traces} on 
the finite dimensional convergence for traces of monomials $A_n^1, \dots, A_n^k$.
The study of the convergence of traces for random permutation matrices 
following the generalized Ewens distribution has been done in \cite{Nikeghbali_Zeindler}. 
From Corollary 3.2 therein, we have that for every $k \geqslant 1$, 
the finite dimensional convergence in law
\begin{equation}
    \label{eq:finite_dim_cv_cycles}
    \left( C_1^{(n)}, \dots, C_k^{(n)} \right) 
    \overunderset{\diff}{n \to \infty}{\longrightarrow} (Y_1, \dots, Y_k)
\end{equation}
holds with $(Y_\ell)_{\ell \geqslant 1}$ independent Poisson random variables 
with parameter $\frac{\theta_\ell}{\ell}r^\ell$.
Using 
\begin{equation*}
    \Tr \left[ A_n^k \right] = \sum_{\ell | k} \ell C_\ell^{(n)}
\end{equation*}
yields the result of Theorem \ref{th:conv_of_traces} by the Cramer-Wold theorem.
\\
\\
Note that the finite dimensional convergence \eqref{eq:finite_dim_cv_cycles} is a consequence of the 
functional equality on generating functions stated as \eqref{eq:generating_function_cycles} 
below which is Theorem 3.1 of \cite{Nikeghbali_Zeindler}:
\begin{equation}
    \label{eq:generating_function_cycles}
    \sum_{n \geqslant 0} h_n^\Theta \E \left[\exp\left( i \sum_{m=1}^{b} s_m C_m^{(n)} \right) \right] t^n 
    = \exp\left( \sum_{m=1}^{b} \frac{\theta_m}{m} 
    \left(\ed^{is_m}-1 \right)t^m \right) G(t) \ .
\end{equation}

\section{Poisson Expression: proof of Theorem \ref{th:poisson_expression}}
\label{sec:poisson_expression}

Recall that $g(z) = \sum_{k \geqslant 1} \frac{\theta_k}{k} z^k$. 
Since $g \in F(r, \gamma, K)$ for some $r > 0$, $\gamma > 0$ and $K \in \C$, 
then $r$ is the convergence radius of $g$, so that 
$\frac{1}{r} = \limsup_k \theta_k^{\frac{1}{k}}$. 
Moreover, since 
\begin{equation*}
    g(z) = - \gamma \log \left( 1 - \frac{z}{r} \right) + K + O(z-r) \ ,
\end{equation*}
one has that $\sum_{k \geqslant 1} \frac{r^k \theta_k}{k} = + \infty$. Indeed, 
for otherwise, by Abel's theorem $g$ would be left continuous at $r$ 
with left limit $\sum_{k \geqslant 1} \frac{r^k \theta_k}{k}$ and not have the singular behaviour above.
Let $(Y_\ell)_{\ell \geqslant 1}$ be a family of 
independent Poisson random variables with parameters 
$ \left(\frac{r^\ell \theta_\ell}{\ell} \right)_{\ell \geqslant 1}$.
Consider the series 
\begin{equation*}
  f(z) = \sum_{k \geqslant 1} \frac{X_k}{k} z^k, \text{ where } 
  X_k = \sum_{\ell | k} \ell Y_\ell \ .
\end{equation*}  
We first show that $f$ is a well-defined function on the open 
disk $\D$ in Proposition \ref{prop:convergence_radius}. 

\begin{proposition}[Radius of convergence for limit function]
    \label{prop:convergence_radius}
    Assume that $g \in F(r, \gamma, K)$ for $r > 0$, 
    $\gamma > 0$ and $K \in \C$.
    Almost surely, the radius of convergence of $f$ is equal to $1$.
\end{proposition}

\begin{proof}
    Let us denote by $r_f$ the convergence radius of $f$. Then, almost surely, 
    \begin{equation*}
        \frac{1}{r_f} =  \limsup_{k \to \infty} \left( \frac{X_k}{k} 
        \right)^{\frac{1}{k}} = \limsup_{k \to \infty} X_k^{\frac{1}{k}}. 
    \end{equation*}
    \begin{itemize}
        \item Let us first show that $r_f \leqslant 1$ almost surely. 
    If the sequence $\left(\frac{r^\ell \theta_\ell}{\ell} \right)_{\ell \geqslant 1}$
    is bounded from above by a constant $C > 0$, then 
    \begin{equation*}
        \prob \left[ Y_\ell = 1 \right] = \frac{r^\ell \theta_\ell}{\ell} 
        \ed^{-\frac{r^\ell \theta_\ell}{\ell}} \geqslant \frac{r^\ell \theta_\ell}{\ell} \ed^{-C}
    \end{equation*}
    so that $\sum_{\ell \geqslant 1} \prob \left[ Y_\ell = 1 \right] = +\infty$ since 
     $\sum_{k \geqslant 1} \frac{r^k \theta_k}{k} = + \infty$. 
    If the sequence is unbounded, for every $j \geqslant 1$, there exists $\ell(j) \geqslant 1$ 
    such that $\frac{r^{\ell(j)} \theta_{\ell(j)}}{\ell(j)} \geqslant j$. Thus, for $j \geqslant 1$,
        \begin{equation*}
        \prob \left[ Y_{\ell(j)} \geqslant 1 \right] 
        = 1 - \ed^{-\frac{r^{\ell(j)} \theta_{\ell(j)}}{\ell(j)}}
         \geqslant 1 - \ed^{-j} .
    \end{equation*}
    In both cases, the lower bound is not summable. 
    Borel-Cantelli's second lemma implies that almost surely, 
    there are infinitely many indices $k \geqslant 1$ such that $Y_{k} \geqslant 1$ 
    and thus $X_{k} \geqslant 1$. Therefore, almost surely $\limsup_{k} X_k^{\frac{1}{k}} \geqslant 1$
    which implies that almost surely $r_f \leqslant 1$ by Hadamard's formula. \\
    \item Let us now show that $r_f \geqslant 1$ almost surely.
    Let $\epsilon > 0$. Since $r$ is the convergence radius of $g$, 
    $\frac{1}{r} = \displaystyle \limsup_{\ell \to \infty} \theta_\ell^{1/\ell}$. Thus,
    there exists $\ell_0$ such that for $\ell \geqslant \ell_0$, 
    \begin{equation*}
        \left| \frac{1}{r} - \sup_{\ell \geqslant \ell_0} \theta_\ell^{\frac{1}{\ell}} \right| 
            \leqslant \frac{\epsilon}{r}
    \end{equation*}
    so that for $\ell \geqslant \ell_0$, 
    \begin{equation*}
        r^\ell \theta_\ell \leqslant (1 + \epsilon)^\ell.
    \end{equation*}
    Define on the same probability space sequences $(Y_\ell)_{\ell \geqslant 1}$ and 
    $(Y'_\ell)_{\ell \geqslant 1}$ having respective parameters 
    $(\frac{r^\ell \theta_\ell}{\ell})_{\ell \geqslant 1}$
    and $ \left(\frac{(1+ \epsilon)^\ell}{\ell} \right)_{\ell \geqslant 1}$,
    such that $Y_\ell \leqslant Y'_\ell$ almost surely for $\ell \geqslant \ell_0$.
    Let us define $(X'_k)_{k \geqslant 1}$ by $X'_k = \sum_{\ell | k} \ell Y'_\ell$. 
    Then, almost surely,
    \begin{equation*}
        X_k \leqslant X'_k + 
        \sum_{\substack{\ell | k \\ \ell \leqslant \ell_0}} \ell(Y_\ell - Y'_\ell) 
        \leqslant X'_k + \chi \ ,
    \end{equation*}
    where $\chi = \sum_{\ell = 1}^{\ell_0} \ell |Y_\ell - Y'_\ell|$, 
    which does not depend on $k$ and which is finite almost surely.
    Therefore, almost surely,
    \begin{equation*}
        \limsup_{k \to \infty} X_k^{\frac{1}{k}} \leqslant 
        \limsup_{k \to \infty} \left( X'_k + \chi \right)^{\frac{1}{k}} 
        \leqslant \limsup_{k \to \infty} \ (X'_k)^{\frac{1}{k}} = 1+\epsilon\ ,
    \end{equation*}
    where we have used that $X'_k + \chi \leqslant X'_k (1 + \chi)$ for the second inequality and that 
    the convergence radius of $\sum_{k \geqslant 1} \frac{X'_k}{k} z^k$ is almost surely 
    $\frac{1}{1+\epsilon}$ using Theorem 2.7, (iii) of \cite{Coste_Bernoulli} 
    for the last equality.
    Therefore, for every $\epsilon > 0$, almost surely, 
    \begin{equation*}
        r_f \geqslant \frac{1}{1+\epsilon} \ .
    \end{equation*} 
    Since $\epsilon > 0$ was arbitrary, we derive that $r_f \geqslant 1$ almost surely.
    \end{itemize}
\end{proof}

\noindent
Since $F(0) = 1$ and that $F \in \mathcal{H}(\D)$, one can consider $\log(F)$ which 
is a well-defined analytic function in a neighborhood of the origin, where $\log$ is the 
principal branch of the logarithm. This 
function coincides with $-f$ so that they are both equal.  
Both functions are well-defined in the unit disk from which one derives the 
desired expression of Theorem \ref{th:poisson_expression}.

\section{Poisson Multiplicative Function}

For the sake of completeness, we provide another representation 
for the limit function of Theorem \ref{th:conv}. 
This representation given in Lemma \ref{lem:infinite_product} 
has the form of an infinite product and 
is inspired from \cite{Coste_Bernoulli} where the exponential 
of a Poisson series appeared in the context of Bernoulli matrices.

\begin{lemma}[Infinite product expression]
    \label{lem:infinite_product}
    Under the assuptions of Theorem \ref{th:conv}, for $z \in \D$, one has 
    \begin{equation}
        \exp(-f(z)) = \prod_{k \geqslant 1} (1- z^k)^{Y_k}.
    \end{equation}
\end{lemma}

\begin{proof}
    The expression above is due to the inversion
    \begin{equation*}
        \sum_{k \geqslant 1} \frac{X_k}{k} z^k 
        = \sum_{\ell \geqslant 1} \ell Y_\ell \sum_{k \geqslant 1} \frac{z^{k\ell}}{k\ell} 
        = -\sum_{\ell \geqslant 1} Y_\ell \log(1 - z^\ell) \ ,
    \end{equation*}
    which can be performed since uniform convergence holds for $z \in \D$. 
\end{proof}

\noindent
As introduced in \cite{Coste_Lambert_Zhu}, 
the expression of $F = \exp(-f)$ is the Poisson analog of the Gaussian holomorphic chaos 
which was first introduced in \cite{Najnudel_Paquette_Simm} 
for the study of the characteristic polynomial of matrices 
from Circular-$\beta$ Ensembles and their coefficients. The Gaussian holomorphic chaos also 
appeared in limit expressions for characteristic polynomials of i.i.d.\ matrices 
\cite{Bordenave_Chafai_Garcia} and 
Gaussian elliptic matrices \cite{Francois_Garcia}. 
It is the Gaussian analog of $F$, replacing Poisson random variables by complex Gaussians. 
This provides an 
example of log-correlated field as correlations for such function $r$ are given 
by $\E[r(z)\overline{r(w)}] = - \log(1-z\overline{w})$.
For the generalized Ewens measure, the correlations are given by the generating 
function $g$ as stated in Lemma \ref{lem:correlations}. 

\begin{lemma}[Correlations of Poisson field]
    \label{lem:correlations}
    For $z, w \in \D$, one has 
    \begin{equation}
        \label{eq:correlation_poisson}
        Cov(f(z), f(w)) = \sum_{a, b \geqslant 1} 
        \frac{1}{ab} g \left(rz^a \overline{w}^b \right) \ .
    \end{equation}
\end{lemma}

\begin{proof}
    Since we want to compute correlations, we must consider the series 
    \begin{equation}
        \label{eq:centered_serie}
        \sum_{k \geqslant 1} \frac{X_k - \E[X_k]}{k} z^k. 
    \end{equation}
    From Proposition \ref{prop:convergence_radius}, we know that the 
    convergence radius of $\sum_{k \geqslant 1} \frac{X_k}{k} z^k$ is at least $1$. 
    Let us check that the same holds for $\sum_{k \geqslant 1} \frac{\E[X_k]}{k} z^k$ 
    so that \eqref{eq:centered_serie} is well-defined for $z \in \D$. 
    Let $\epsilon > 0$. As in Proposition \ref{prop:convergence_radius}, there 
    exists $\ell_0 \geqslant 1$ such that for $\ell \geqslant \ell_0: $ 
    $\theta_\ell r^\ell \leqslant (1+\epsilon)^\ell$. Thus,
    \begin{equation*}
        \E[X_k] = \sum_{\ell | k} r^\ell \theta_\ell 
        \leqslant \sum_{\substack{\ell | k \\ \ell \leqslant \ell_0}} (r^\ell \theta_\ell - (1+\epsilon)^\ell) 
        +  \sum_{\ell | k} (1+\epsilon)^\ell
    \end{equation*}
    so that $|\E[X_k]| \leqslant (c + 1) \tau_k$ where 
    $\tau_k = \sum_{\ell | k} (1+ \epsilon)^k$ 
    and $c = \sum_{ \ell \leqslant \ell_0} |\theta_\ell r^\ell - (1+\epsilon)^\ell|$.
    The latter implies that $\limsup_k |\E[X_k]|^{1/k} \leqslant \limsup \tau_k^{1/k} = 1+\epsilon$
    from which one derives that the convergence radius of $\sum_{k \geqslant 1} \frac{\E[X_k]}{k} z^k$ 
    is greater than $\frac{1}{1+\epsilon}$. Since $\epsilon$ was arbitrary, the convergence radius 
    is greater or equal to one so that \eqref{eq:centered_serie} is well-defined for $z \in \D$.
    For $z, w \in \D$,
    \begin{align*}
        Cov(f(z), f(w)) &= \sum_{k, h} \frac{z^k \overline{w}^h}{kh} \sum_{\substack{i | h \\ j | k}} 
        ij Cov(Y_i, Y_j) \\
        &= \sum_{k, h} \frac{z^k \overline{w}^h}{kh} \sum_{\ell | k, \ell | h} 
        \ell^2 Var[Y_\ell] \\
        &= \sum_{k, h} \frac{z^k \overline{w}^h}{kh} \sum_{\ell | k, \ell | h} 
        \ell \theta_\ell r^\ell \\
        &= \sum_{\ell \geqslant 1} \frac{\theta_\ell r^\ell}{\ell} \sum_{a, b \geqslant 1} 
        \frac{z^{a \ell} \overline{w}^{b \ell}}{ab} \\
        &= \sum_{a, b \geqslant 1} \frac{1}{ab}  \sum_{\ell \geqslant 1} 
        \frac{\theta_\ell}{\ell} \left( rz^a \overline{w}^b \right)^\ell 
        = \sum_{a, b \geqslant 1} \frac{1}{ab} \, g \left( rz^a \overline{w}^b \right).
    \end{align*}

\end{proof}

\begin{remark}
    In the case of uniform permutations \cite{Coste_Lambert_Zhu} 
    or even for Ewens random permutations, that is, $\theta_k = \theta$ 
    for some $\theta > 0$, one has $r=1$ and $g(z) = - \theta \log(1 - z)$ 
    so that 
    \begin{equation*}
        Cov(f(z), f(w)) = - \theta \sum_{a, b \geqslant 1} 
        \frac{1}{ab} \log \left(1- z^a \overline{w}^b \right) \ ,
    \end{equation*}
    which is the analog of the log-correlations obtained for the Gaussian holomorphic chaos. 
    In general, the correlations for arbitrary sequences $\Theta$ are given by $g$. 
    Moreover, the expectation of the limit can be expressed using $G$ for any $z \in \D$,
    \begin{equation*}
        \E \left[ \prod_{k \geqslant 1} (1- z^k)^{Y_k} \right] 
        = \prod_{k \geqslant 1} \ed^{-\theta_k \frac{r^kz^k}{k}} = \frac{1}{G(rz)} \ .
    \end{equation*}
\end{remark}

\printbibliography

\end{document}